\documentclass{amsart}
\usepackage{amssymb,amsmath,amsthm,amsfonts,epsfig,graphicx}
\usepackage[utf8]{inputenc}
\usepackage[T1]{fontenc}
\usepackage[english]{babel}
\usepackage{hyperref}

\hypersetup{
    colorlinks,
    citecolor=black,
    filecolor=black,
    linkcolor=black,
    urlcolor=black
}

\theoremstyle{plain}
\newtheorem{teo}{Theorem}[section]
\newtheorem{cor}[teo]{Corollary}
\newtheorem{lem}[teo]{Lemma}
\newtheorem{prop}[teo]{Proposition}
\theoremstyle{definition}
\newtheorem{rmk}[teo]{Remark}
\newtheorem{df}[teo]{Definition}
\newtheorem{exa}[teo]{Example}

\DeclareMathOperator{\ann}{Ann}
\DeclareMathOperator{\sign}{Sign}

\DeclareMathOperator{\spec}{spec}
\DeclareMathOperator{\specm}{specm}

\DeclareMathOperator{\diam}{diam}
\DeclareMathOperator{\clos}{clos}

\DeclareMathOperator{\dist}{dist}

\DeclareMathOperator{\card}{card}
\newcommand{\rad}[1]{\sqrt{#1}}

\newcommand{\U}{{\mathcal U}}
\newcommand{\A}{{\mathcal A}}
\newcommand{\B}{{\mathcal B}}

\newcommand{\V}{{\mathcal V}}
\newcommand{\W}{{\mathcal W}}
\newcommand{\m}{\mathfrak{m}}

\newcommand{\bigtimes}{\times}
\renewcommand{\ll}{\prec}

\newcommand{\I}{\mathcal I}
\newcommand{\J}{\mathcal J}
\newcommand{\OO}{\mathcal O}

\newcommand{\R}{\mathbb R}
\newcommand{\Q}{\mathbb Q}
\newcommand{\K}{\mathcal K}

\newcommand{\Z}{\mathbb Z}
\newcommand{\N}{\mathbb N}
\renewcommand{\epsilon}{\varepsilon}

\title{Expansivity on commutative rings}
\date{\today}
\author[A.~Artigue]{Alfonso Artigue}
\address{Departamento de Matemática y Estadística del Litoral, Universidad de la Rep\'ublica, Salto, Uruguay.}
\email{artigue@unorte.edu.uy}

\author[M.~Haim]{Mariana Haim}
\address{Centro de Matem\'atica\\
Facultad de Ciencias, Universidad de la Rep\'ublica\\
Montevideo, Uruguay.}
\email{negra@cmat.edu.uy}

\begin{document}

\begin{abstract}
In this article we extend the notion of expansivity from topological dynamics to automorphisms of commutative rings with identity. 
We show that
a ring admits a 0-expansive automorphism if and only if it is a finite product of local rings. 
Generalizing a well known result of compact metric spaces,
we prove that if a ring admits a positively expansive automorphism then it admits finitely many maximal ideals. 
We prove its converse for principal ideal domains. 
We also consider the topological expansivity induced, in the spectrum of the ring with the Zariski topology, 
by an automorphism and some consequences are derived.
\end{abstract}
\maketitle

\section{Introduction}

Given a compact metric space $(X,\dist)$ we say that 
a homeomorphism $h\colon X\to X$ is \emph{expansive} if there is 
$\delta>0$ such that if $x,y\in X$, $x\neq y$, then $\dist(h^n(x),h^n(y))>\delta$ 
for some $n\in \Z$. 
In \cite{AH} the reader can find several results on expansive homeomorphisms 
that show the important role played by expansivity in dynamical systems.
Since \cite{AdKoMc,KR} it is known that expansivity can be expressed independently of the metric. 
In \cite{KR} it is shown that expansivity is equivalent to the existence of a \emph{topological generator}. 
In \cite{AAM} this notion is generalized to topological spaces and 
examples on non-Hausdorff spaces are given. 
The key is to consider the action of $h$ on the open covers of $X$.
The purpose of this article is to extend the notion of expansivity to 
an algebraic context. 


If $X$ is a metric space, we consider
the commutative unital ring $C(X)$ of continuous functions $f\colon X\to\R$. 
A homeomorphism $h$ of $X$ induces an automorphism of $C(X)$.
Open subsets of $X$ are naturally associated to ideals of $C(X)$ and 
open covers of $X$ give rise to algebraic generators of $C(X)$. 
In this way, the dynamical notions that can be expressed in terms of open covers, 
can be translated to automorphisms of rings. 
The rings that we consider are unital and commutative.

Our first result comes from the following topological fact. 
If $X$ is a finite set then there is $\delta>0$ such that $\dist(h^n(x),h^n(y))>\delta$ for all $x\neq y$ and all $n\in\Z$, in particular for $n=0$. From a dynamical point of view this example is trivial, the points are already separated at instant $n=0$. 
This property of the topological space has nothing to do with the dynamics, but sometimes we refer to it as $0$-expansivity of each homeomorphism.
As a particular case, on a finite set, even the identity is expansive. 
There is the notion of $0$-expansivity in the algebraic framework, and it depends on the ring and not on the automorphism. In Theorem \ref{teoChar0Exp}, we prove that $0$-expansivity 
holds 
precisely for finite products of local rings. However, there are some differences in this new context. In particular, 
we give an example of a ring such that the identity is expansive but not $0$-expansive (see Remark \ref{elejemplo}).

Our second result is related to positive expansivity, i.e., the separation occurs at some $n\geq 0$. 
It is known that if a compact metric space admits a positively expansive homeomorphism then it is finite, see for example \cite{CK}. 
 In Theorem \ref{teoPosExpFinMax} we prove that if $R$ admits a positively expansive automorphism then $R$ has 
finitely many maximal ideals. Its converse is proved for principal ideal domains in Theorem \ref{detopolandia}.

There is a classical functor from the category of commutative rings to the category of topological spaces that 
associates to a ring its prime spectrum (the set of all primes ideals of the ring) endowed with the so called Zariski topology.
In \S\ref{secZariski}, we show that algebraic expansivity of a ring automorphism is strictly stronger than topological expansivity of its spectrum.
Naturally, this functor shed light on some proofs during the research. These links are explained in \S\ref{secExtClos} and lead to a counterexample for the converse of
Theorem \ref{teoPosExpFinMax} that we present in \S\ref{secSpSp}. 
Moreover, we show that there is a ring with finite maximal spectrum and such that the identity is not positively expansive.


This article is organized as follows. 
In \S\ref{secTopoRing} we introduce the main notions of this paper, 
prove some basic properties and state the adequacy in the topological context of our algebraic approach.
In \S\ref{secStrongExp} we consider some strong forms of expansivity and prove Theorems \ref{teoChar0Exp}, \ref{teoPosExpFinMax} and \ref{detopolandia}. 
Finally, in \S\ref{secZ} we consider topological expansivity on the prime and maximal spectrum of a ring.
We thank Ali Barzanouni 
	for his
	kind and useful comments on a preliminary version of this paper.

\section{Topological and algebraic expansivity}
\label{secTopoRing}
In this section we present the notion of expansivity of 
homeomorphisms and its translation to automorphisms of commutative rings. Under this translation, open covers correspond to generators of rings. 

The following is a general definition that will be used in both contexts (mainly for $\A,\B$ open covers and $\A,\B$ generators of a ring). 
\begin{df}
Let $\A$ and $\B$ be families of sets. We say that $\A$ \emph{refines} $\B$ and write $\A\prec \B$ 
if for any $A\in \A$, there is some $B\in \B$ such that $A\subseteq B$.
\end{df}

\subsection{Topological expansivity} 
Let us start introducing some notation, explaining how to express expansivity without the metric and recalling some known results.
We assume that $X$ is a compact topological space and that the open covers of $X$ are finite.
For each $i=1,\dots,n$ let $\U_i$ be an open cover of $X$. 
Following \cite{AdKoMc,KR}, we consider the open cover
$
\wedge_{i=1}^n \U_i 
$
defined as the family of all intersections of the form $U_1\cap U_2\cap \cdots \cap U_n$, with $U_i\in \U_i$.
If $h\colon X\to X$ is continuous then $h^{-1}(\U)=\{h^{-1}(U)\mid U\in \U\}$ is an open cover of $X$.

\begin{rmk}
\label{rmkOpCovIdemp}
If $\U$ is an open cover of cardinality $k$, we define
$\V=\wedge_{i=1}^k\U$. It is easy to see that $\V\ll\U$ and $\V\wedge \V=\V$. 
That is, every open cover can be refined by an \emph{idempotent} open cover.
\end{rmk}

The key to translate expansivity of metric spaces to the language of open covers is the Lebesgue number. 
We say that $\sigma>0$ is a \emph{Lebesgue number} for an open cover $\U$ if 
$\dist(x,y)<\sigma$ implies that there is $U\in\U$ such that $x,y\in U$. 
Every open cover of a compact metric space has a Lebesgue number (see \cite[Theorem 26, p. 154]{K}). 

\begin{prop}
\label{propEquivExpCub}
If $X$ is a compact metric space and $h\colon X\to X$ is a homeomorphism then the following are equivalent:
\begin{enumerate}
\item $h$ is expansive,
\item there exists an open cover $\U$ such that for any other open cover $\V$ we have
$
\wedge_{|i|\leq N} h^{-i}(\U)\ll \V
$
for some $N\geq 0$.
\end{enumerate}
\end{prop}

\begin{proof}
Given an expansivity constant $\delta$, take an open cover $\U$ such that $\diam(U)<\delta$ for all $U\in\U$. 
Let $\V$ be an open cover and take $\sigma>0$ a Lebesgue number for $\V$.
Arguing by contradiction, suppose that for each $N$ there are open sets 
$U_i\in\U$, $|i|\leq N$, 
such that $\cap_{|i|\leq N}h^{-i}(U_i)$ is not contained in any $V\in\V$. 
Thus, there are $x_N,y_N\in U_0$ such that $\dist(h^i(x_N),h^i(y_N))<\delta$ for all $|i|\leq N$ and 
$\dist(x_N,y_N)\geq\sigma$ (otherwise, there would be $V\in\V$ containing these points). 
As $X$ is compact there are limit points $x_*,y_*$ of subsequences of $x_N,y_N$. 
We conclude that $\dist(x_*,y_*)\geq\sigma$ (i.e., $x_*\neq y_*$) and 
$\dist(h^i(x_*),h^i(y_*))\leq\delta$ for all $i\in\Z$. 
This contradicts that $\delta$ is an expansivity constant.

Conversely, take an open cover $\U$ of $X$ such that for any other open cover $\V$, we have
$\bigcap_{i=-N}^N h^n(\U)\ll \V$ for some $N$.
Let $\delta>0$ be a Lebesgue number for $\U$. 
We will show that $\delta$ is an expansivity constant. 
Take $x,y\in X$ such that $\dist(x,y)=\epsilon\in (0,\delta)$. 
Consider an open cover $\V$ such that $\diam(V)<\epsilon$ for all $V\in\V$.
We know that there is $N$ such that 
$\bigcap_{i=-N}^N h^n(\U)\ll \V$. 
If $\dist(h^i(x),h^i(y))<\delta$ for all $|i|\leq N$, 
then for each $i$ we can take $U_i\in\U$ such that $h^{-i}(x),h^{-i}(y)\in U_i$. 
Then $x,y\in U_*=\bigcap_{i=-N}^N h^i(U_i)\in\bigcap_{i=-N}^N h^n(\U)$, which is a contradiction, since $x,y\in U_*\subset V$ for some $V\in \V$ and $\dist(x,y)=\epsilon>\diam(V)$.
\end{proof}

In light of Proposition \ref{propEquivExpCub},
a homeomorphism $h$ of a topological space is \emph{expansive} if  
there exists an open cover $\U$ such that for any other open cover $\V$, there is some $N\geq 0$ such that 
$
\wedge_{|i|\leq N} h^{-i}(\U)\ll \V.
$
It was called \emph{refinement expansivity} in \cite{AAM}.
We say that $h$ is \emph{positively expansive} if
there exists an open cover $\U$ of $X$ such that for any other open cover $\V$, there is some $N\geq 0$ such that
$
\wedge_{i=0}^N h^{-i}(\U)\ll \V.$\footnote{Note that the two topological notions of expansivity can be defined for general metric spaces. However, in the non compact case, they are not equivalent. Indeed, the existence of a refinement expansive homeomorphism on $X$ implies compactness of $X$ (\cite[Corollary 3.10]{AAM}), while the definition with the metric does not (for example $f(x)=2x$ in $\R$ with the usual distance is expansive).}

In what follows, we list some results that will be considered in the generalization of topological expansivity for
commutative rings.

\begin{prop}\cite[Lemma 3.19]{AAM}
	\label{lemPosExp}
	A homeomorphism $h\colon X\to X$ of a compact topological space is positively expansive if and only if 
	there is an open cover $\U$ such that for every open cover $\V$ there is $n\geq 0$ 
	such that $h^{-n}(\U)\ll\V$.
\end{prop}

An open cover $\U$ of $X$ is $\ll$-\emph{minimal} if $\U\ll\V$ for every open cover $\V$.

\begin{prop}
	\label{propPosExpHomeoMinCov}
	The identity of a topological space $X$ is positively expansive if and only if 
	there is a $\ll$-minimal open cover.
\end{prop}

\begin{proof}
	It is clear that the open cover $\U$ given by Proposition 
	\ref{lemPosExp} is $\ll$-minimal if the identity is positively expansive.
	For the converse, notice that a $\prec$-minimal open cover 
		is an expansivity cover for the identity.
\end{proof}

\begin{prop}\cite[Theorem 3.20]{AAM}
\label{propUtzT1}
If $X$ is a compact $T_1$-space that admits a positive expansive homeomorphism then 
$X$ is a finite set.
\end{prop}

%
In Example \ref{exaX3Exp} we will show that Proposition \ref{propUtzT1} is not true if the space is not assumed to be $T_1$.

%
%

\subsection{Generators}
In this section we introduce generators of rings in analogy with open covers of topological spaces.
Throughout this article, $R$ will denote a unital commutative ring. 
A finite set $\I$ of ideals of $R$ is a \emph{generator} if $R=\sum_{I\in\I}I$.
Given two generators $\I,\J$ define their product as
\[
\I\J=\{IJ\mid I\in\I,J\in\J\}.
\]
The role of the operation $\wedge$ between open covers 
is played by the product of generators.\footnote{Since the intersection of ideals is an ideal, it also makes sense to use 
$\wedge$ as an operation between generators. However, we choose the product which seems, 
in view of Proposition \ref{perp}, (10), the natural operation 
in this algebraic context.}
The next result summarizes some basic properties of generators that will be needed in what follows.
A ring is \emph{local} if it has a unique maximal ideal. 

\begin{prop}
\label{propGenBasic}
The following properties hold:
\begin{enumerate}
\item if $\I \prec \J$ and $\I' \prec \J'$ then $\I\I'\prec \J\J'$,
\item if $\I_1,\dots,\I_n$ are generators then $\Pi_{i=1}^n \I_i$ is a generator,
\item if $\alpha\colon R\to R$ is an automorphism and $\I$ is a generator then the set 
$\alpha^{-1}(\I)=\{\alpha^{-1}(I)\mid I\in \I\}$  is a generator, 
\item $R$ is local if and only if $R\in\I$ for every generator $\I$ of $R$. 
\end{enumerate}
\end{prop}

\begin{proof}
An element in $\I\I'$ is an ideal $II'$, with $I\in \I, I'\in \I'$. 
There exist $J\in \J$ and $J'\in \J'$ 
such that $I\subseteq J$ and $I' \subseteq J'$. 
Therefore $II'\subseteq JJ'\in \J\J'$.

To prove that $\Pi_{i=1}^n \I_i$ is a generator note that 
the distributivity of the product in a ring allows us to generate the unity with $\Pi_{i=1}^n \I_i$.

To prove that $\alpha^{-1}(\I)$ is a generator notice that each $\alpha^{-1}(I)$ is an ideal.
If $1\in R$ is the unity and $1=\sum_{I\in\I} a_I$ with $a_I\in I$, then 
$1=\sum_{I\in\I} \alpha^{-1}(a_I)$. This proves that $\alpha^{-1}(\I)$ generates $R$. 

If $R$ is local and $\m$ is its maximal ideal, every ideal $J\neq R$ is contained in $\m$. Thus, a family of ideals not containing $\{R\}$ will generate an ideal included in ${\mathfrak m}$. Therefore, $R$ belongs to any generator. 
Conversely, if $\m_1,\m_2$ are different maximal ideals, then $\{\m_1,\m_2\}$ is a generator not containing $R$.
\end{proof}

The next example shows an important difference between open covers of topological spaces and generators of rings. 
The example will be also used later.
Let $\Z_{2,3}$ be the subring of $\Q$ of rational numbers whose reduced expression $\frac{m}{n}$ is such that $n$ is neither even nor a multiple of $3$.

\begin{prop}
\label{propQ23basic}
The ring $\Z_{2,3}$ satisfies that: 
\begin{enumerate}
\item its ideals are principal,
\item its maximal ideals are $(2)$ and $(3)$,
\item its prime ideals are $(0),(2)$ and $(3)$,
\item a family of ideals $\I$ is a generator if and only if $(2^a),(3^b)\in \I$ for some $a,b\geq 0$,
\item its idempotent generators\footnote{A generator $\I$ is \emph{idempotent} if $\I^2=\I$.} are $\{\Z_{2,3}\}$ and $\{\Z_{2,3},\{0\}\}$.
\end{enumerate}
\end{prop}

\begin{proof}
Take a non zero ideal $I$ and $\frac{m}{n}\in I$ in its reduced expression. Multiplying by $\frac{n}{1}$ we get $m\in I$. Consider the minimal positive integer $m_0$ in $I$. It can be shown that $I=(m_0)$. 

Any integer wich is coprime with $2$ and with $3$ is invertible, and therefore generates the ideal $R$.
So every ideal is of the form $(2^a3^b)$ with $a,b\geq 0$ and the maximals are $(2)$ and $(3)$. 
\end{proof}

In Remark \ref{rmkOpCovIdemp} we explained that every open cover can be refined by an idempotent open cover. 
In the ring $\Z_{2,3}$, the generator $\I=\{(2),(3)\}$ can not be refined by an idempotent generator. 

\subsection{Expansive automorphisms}
\label{secExpAuto}
We say that an automorphism $\alpha\colon R\to R$ of a commutative unital ring
is an \emph{expansive automorphism} if 
there is a generator $\I$ such that 
for every generator $\J$ there is $N\geq 0$ such that
\[
\Pi_{|i|\leq N}\alpha^{-i}(\I)\prec \J.
\]
We will say that $\I$ is an $\alpha$-\emph{generator} of expansivity.
Similarly, we say that $\alpha\colon R\to R$ is \emph{positively expansive} if there is a generator $\I$ such that 
for every generator $\J$ there is $N\geq 0$ such that
\[
\Pi_{i=0}^N\alpha^{-i}(\I)\prec \J.
\]
As a first example, note that every automorphism of a ring with finitely many ideals is expansive 
Indeed, as there are finitely many generators, the product of them is a generator of expansivity. In particular, on a finite ring, every automorphism is expansive.

In what follows we will derive some fundamental properties of expansive automorphisms 
extending well known result from topological dynamics. 

\begin{prop}
The following properties hold: 
\begin{enumerate}
\item every positively expansive automorphism is expansive,
\item if $id\colon R\to R$ is expansive then it is positively expansive,
\item an automorphism $\alpha\colon R\to R$ is expansive if and only if $\alpha^n$ is expansive for all $n\in\Z$, $n\neq 0$.
\end{enumerate}
\end{prop}

\begin{proof}
For every generator $\I$ and every automorphism $\alpha$ it holds that 
$\Pi_{|i|\leq N}\alpha^{-i}(\I)\prec \Pi_{i=0}^N\alpha^{-i}(\I)$. 
Therefore, if $\alpha$ is positively expansive with expansive generator $\I$ then 
$\alpha$ is expansive with the same expansive generator.
If $\alpha$ is the identity then
$\Pi_{|i|\leq N}id^{-i}(\I)= \Pi_{i=0}^{2N}id^{-i}(\I)=\I^{2N+1}$,
which proves that expansivity and positive expansivity are equivalent for the identity.

If $\alpha$ is expansive, consider an expansive generator $\I$. 
It is clear that $\I$ is also an expansive generator for $\alpha^{-1}$. 
Thus, we assume that $n>1$. Let $\J=\Pi_{|i|\leq n}\alpha^i(\I)$.
Let $\K$ be a generator and from the expansivity of $\alpha$ take $N$ such that 
$\Pi_{|i|\leq N}\alpha^i(\I)\prec \K$. 
Assuming that $nL\geq N$ we have that 
\[
\Pi_{|j|\leq L}\alpha^{nj}(\J)\prec\Pi_{|i|\leq N}\alpha^i(\I)\prec \K.
\]
This proves that $\J$ is an expansive generator for $\alpha^n$. 
Conversely, it is easy to see that if $\I$ is an expansive generator for $\alpha^n$ then $\I$ is also 
an expansive generator for $\alpha$. 
\end{proof}

The following example is generalized later by Theorem \ref{detopolandia} to principal ideal domains.

\begin{exa}[The ring of integers $\Z$] As automorphisms preserve the unit, the unique automorphism of $\Z$ is the identity. It is not expansive. Indeed, any generator of $\Z$ contains $\Z$ or contains two principal ideals whose generators are coprime. Now, assume we have a generator of $id$-expansivity $\K$. Take another generator $\J=\{(p), (p')\}$ with $p,p'$ coprime. If $\K$ contains $R$, it is clear that $R^n=R$ is not included in any of the ideals of $\J$. If $\K$ does not contains $R$, it contains some $(d)$; chosing $p,p'$ coprime not dividing $d$, we get that $(d)^n$ can not be contained in any ideal of $\J$ for each $n\in \N$. 
\end{exa}

The properties of local rings sketched in the next example are the key of the proof of Theorem \ref{teoChar0Exp}.

\begin{exa}[Local rings] If $R$ is local, any generator contains $R$, so any automorphism is expansive ($\{R\}$ is a generator of expansivity). Moreover, we can take an homogeneous $N=0$ in the definition of expansivity and it will do the job.
\end{exa}

The following example is simple but important to illustrate some particular properties of algebraic expansivity.

\begin{exa}[The ring $\Z_{2,3}$ of Proposition \ref{propQ23basic}] 
\label{partedelejemplo}
Its only automorphism is the identity and a generator is a set of ideals containing the whole ring or containing two ideals of the form $(2^m)$ and $(3^n)$, with $m,n>0$. We deduce that $\{(2),(3)\}$ is a generator of $id$-expansivity and that $id$ is positive expansive. 
\end{exa}

\subsection{Equivalence in the topological framework}
Let $C(X)$ be the ring of continuous functions from a compact metric 
 space $X$ to $\R$.
Consider on one side subsets of $X$ and on the other subsets of $C(X)$. There is a correspondence given as follows:
\begin{itemize}
	\item for $A$ a subset of $X$, take $A^\perp$ to be the set of functions vanishing in every $x\in A$,
	\item for $S$ a subset of $C(X)$, take $S^\perp$ to be the set of points of $X$ where every $f\in S$ vanishes.
\end{itemize}
For $x\in X$ define $\m_x=\{x\}^\perp$.

\begin{prop}\label{perp}The following properties hold:
\begin{enumerate}
	\item $A^\perp$ is always an idempotent ideal,
	\item $S^\perp$ is always a closed set,
	\item both $\perp$ invert inclusion and $(A^\perp)^\perp\supseteq A$, $(S^\perp)^\perp\supseteq S$,
	\item the correspondence $x\mapsto \m_x$ is bijective from $X$ onto the set of maximal ideals in $C(X)$,
in particular, in $C(X)$ every maximal ideal is idempotent,
	\item if $\U$ is a finite open cover of $X$, then the family of ideals
	$$
	\I_{\U}=\{I_U=(U^c)^\perp \mid u\in \U\}
	$$
is an idempotent generator of $C(X)$,
	\item if $\I$ is a generator of $C(X)$, then the family of open subsets
	$$
	\U_{\I}=\{(I^\perp)^c\mid I\in \I\}
	$$
	is an open cover of $X$,
	\item if $\U \ll \V$, then $\I_{\U} \prec \I_{\V}$,
	\item if $\I\prec \J$, then $\U_{\I} \ll \U_{\J}$,
	\item if $A,B$ are subsets of $X$, then $(A\cup B)^\perp =A^\perp \cap B^\perp$,
	\item if $S,T$ are subsets of $C(X)$, then $(S\cdot T)^\perp =S^\perp \cup T^\perp$.
\end{enumerate}
\end{prop}

\begin{proof}
It is clear that 
$A^\perp$ is an ideal and
$S^\perp$ is closed.
To prove that $A^\perp\subseteq [A^\perp]^2$ take $f\in A^\perp$ and
notice that 
$\sqrt{|f|},\sqrt{|f|}\sign(f)\in A^\perp$ and 
$f=[\sqrt{|f|}][\sqrt{|f|}\sign(f)]$.
Item (3) follows from the definitions and a proof of (4) can be found in \cite[Theorem 4.9]{GJ}.


In order to prove (5) consider, for each $U\in\U$, the function $f_U(x)=\dist(x,X\setminus U) \in I_U$.
Also, consider $f\in C(X)$ given by $f(x)=\sum_{U\in\U}f_U(x)$ is positive and then invertible. As $\I$ generates an invertible element, it generates every element in $C(X)$.

For (6) take a generator $\I$ of $C(X)$ and $x\in X$. We will prove that $x\not \in I^\perp$ for some $I\in \I$. If this is not the case, then for all $I\in \I$ and all $f\in I$, we have $f(x)=0$, and then $\I$ would generate an ideal included in $\m_x$.

Assertions (7) and (8) come from the fact that when comparing open (instead of closed) sets and ideals, the inclusion is preserved.
Direct proofs lead to the last two assertions.
\end{proof}

Proposition \ref{perp} gives us a way of comparing topological notions of the space $X$ to algebraic notions of the ring $C(X)$. We use it in what follows to compare topological and algebraic expansivity. 

Given a homeomorphism $h$ of $X$ define the automorphism $\alpha_h\colon C(X)\to C(X)$ as 
$\alpha_h(f)=f\circ h^{-1}$.
\begin{rmk}\label{rmk perp}
Let $A$ be a closed subset of $X$. Then $h(A)^\perp=\alpha_h(A^\perp)$. Indeed,
$$
\begin{array}{ll}
(h(A))^\perp&=\{f \in C(X) \mid f(y)=0\ \forall y\in h(A)\}\\
&=\{f \in C(X) \mid f(h(x))=0\ \forall x\in A \}\\
&=\{f \in C(X) \mid (\alpha_{h^{-1}}(f))(x)=0\ \forall x\in A\}\\
&=\{f \in C(X) \mid \alpha_{h^{-1}}(f) \in A^\perp\}\\
&=\{\alpha_h(f) \in C(X) \mid f\in A^\perp\}\\
&=\alpha_h(A^\perp).
\end{array}
$$
Analogously, it holds that  $h(S^\perp)=(\alpha_h(S)^\perp)$ for $S\subset C(X)$.
\end{rmk}

Using parts (5) to (10) of Proposition \ref{perp} and Remark \ref{rmk perp} with the fact that if $I,J$ are ideals, then $IJ \subseteq I\cap J$, we obtain the following result.

\begin{teo}
\label{teoGenIdemp}
For a homeomorphism $h\colon X\to X$ of a compact metric space the following statements are equivalent:
\begin{enumerate}
\item $h$ is an expansive homeomorphism,
\item $\alpha_h\colon C(X)\to C(X)$ is an expansive automorphism,
\end{enumerate}
The same is true for positive expansivity.
\end{teo}

\begin{proof}
Assume that $h\colon X\to X$ is expansive. 
By Proposition \ref{propEquivExpCub} there is an expansivity cover $\U$.
Consider $\I_{\U}$ the generator associated to $\U$ as in Proposition \ref{perp}.
Given any generator $\J$ of $C(X)$
consider the open cover $\V_{\J}$
and another open cover $\W$ such that $\{\clos(W)\mid W\in \W\}\prec \V$.
By Proposition \ref{propEquivExpCub} there is $N>0$ such that $\wedge_{|i|\leq N} h^{-i}(\U)\prec\W$.

We will show that $\Pi_{|i|\leq N}\alpha_h^{-i}(\I_{\U})\prec \J$.
Given a finite sequence $U_i\in\U$, $|i|\leq N$, there are $V\in\V_{\J}$ and $W\in \W$ such that 
$$\cap_{|i|\leq N}h^{-i}(U_i) \subseteq W\subseteq \clos(W)\subseteq V.$$ Note that this is possible by refining $\V$ by a cover of balls $\W'$ and then for each ball $B(x,\varepsilon')\in \W'$, considering $B(x,\varepsilon)$ with $\varepsilon <\varepsilon'$; the new balls $B(x,\varepsilon')$ form the $\W$ we need. 

Given $g\in \Pi_{|i|\leq N}\alpha_h^{-i}(I_{U_i})$ we have that
$g(x)=0$ for all $x\notin W$. 
Let $J\in \J$ such that $V=(J^\perp)^c$. Note that for any $x\in V$, there is some $f\in J$ such that $f(x)\neq 0$. In particular this holds for any $x\in \clos(W)$ and there is (by preservation of the sign) an open cover $\OO
=\{O_\lambda\mid \lambda\in \Lambda\}$ of $\clos(W)$ such that for each $\lambda\in \Lambda$, there is a function $f_\lambda\in J$ such that $f_\lambda(x)\neq 0, \forall x\in O_\lambda$. Consider a finite subcover $\{O_{\lambda_1},\cdots,O_{\lambda_k}\}$ of $\OO$. The function $f=f_{\lambda_1}^2+\cdots f_{\lambda_k}^2\in J$ does not vanish in any element of $\clos(W)$.
Take $\tilde g\in C(X)$ such that 
$\tilde g(x)=\frac{g(x)}{f(x)}$ if $x\in\clos(W)$ and $g(x)=0$ if $x\in V^c$. As $\clos(W)$ and $V^c$ are disjoint closed sets, $\tilde g$ can be extended to a continuous function in $X$. 
We obtain that $g=\tilde gf$ and then $g\in J$.

For the converse, take a (finite) generator $\I=\{I_1,I_2,\cdots,I_k\}$ that makes $\alpha_h$ an expansive automorphism of $C(X)$ and let $1=f_1+f_2+\cdots +f_k$, with each $f_i\in I_i$. Observe that $U_i=\{x\in X\mid f_i(x)\neq 0\}$, with $i\in \{1,2,\cdots,k\}$, defines an open cover of $X$. We will prove that it is an expansivity cover for $h$.

Let $\V$ be any open cover of $X$ and consider, for each $V\in \V$, the ideal $J_V=(V^c)^\perp$. We know by Proposition \ref{perp} that $\J=\{J_V\mid V\in \V\}$ is a generator of $C(X)$ and therefore 
$$
\Pi_{|i|\leq N}\alpha_h^{-i}(I_i)\prec \J,
$$
so $\Pi_{|i|\leq N}\alpha_h^{-i}(f_i) \in J_{i^*}$ for some $i^*$.

Now, take $x_0\in \bigcap_{|i|\leq N}h^{-i}(U_i)$, we have $h^i(x_0)\in U_i$ and therefore, for all $i$, we get $\alpha_h^{-i}(f_i)(x_0)=f_i(h^i(x_0))\neq 0$.
Then,  
$$
\left (\Pi_{|i|\leq N}\alpha_h^{-i}(f_i)\right )(x_0)\neq 0.
$$ 
As $f(x)=0$ for all $f\in J_{i^*}$ and $x\in V_{i^*}^c$, we deduce that $x_0\in V_{i^*}$. 

The proof for positive expansivity is similar.
\end{proof}

In the metric framework, it is clear that expansivity is preserved by disjoint union and by restriction to closed sets (sets need to be close in order that the definition of expansivity via covers hold; for the general non metric case, see \S.\ref{secExtClos}). These facts are translated to the algebraic context by observing that $C(X\sqcup Y)\cong C(X)\times C(Y)$ and that $C(Y)\cong \frac{C(X)}{Y^\perp}$ for $Y$ a closed subspace of $X$. We get that expansivity for a ring automorphism is preserved under products and under quotients. We present self-contained proofs of these two facts. 

For the next result consider automorphisms $\alpha_i\colon R_i\to R_i$ of the rings $R_i$, $i=1,\dots, n$. 
The product automorphism $\bigtimes_{i=1}^n\alpha_i\colon\bigtimes_{i=1}^nR_i\to \bigtimes_{i=1}^nR_i$ 
is defined by
\[
\left(\bigtimes_{i=1}^n\alpha_i\right)(r_1,\dots,r_n)=(\alpha_1(r_1),\dots,\alpha_n(r_n)).
\]
\begin{rmk}
	\label{rmkIdealProd}
	For every ideal $I\subset R_1\times\dots\times R_n$ there are ideals $I_i\subset R_i$ such that 
	$I=I_1\times\dots\times I_n$.
	See \cite[Exercise 20, p. 135]{H}.
\end{rmk}

\begin{prop}
	\label{propProdExpAnillos}
	The product automorphism $\alpha_1\times\dots\times\alpha_n$ is expansive if and only if each $\alpha_i$ is expansive.
\end{prop}

\begin{proof}
	Arguing by induction, it is enough to consider $n=2$. 
	Suppose that $\I_1, \I_2$ are generators of expansivity for $\alpha_1$ and $\alpha_2$ respectively. 
	Consider the following generator of $R_1\times R_2$
	\[
	\I=\{I_1\times I_2\mid I_1\in \I_1,I_2\in\I_2\}.
	\]
	Take any other generator $\J$ of $R_1\times R_2$ and consider the sets of ideals
	$$
	\begin{array}{c}
	\J_1=\{J_1\mid \mbox{ there is }J_2 \mbox{ ideal in } R_2 \mbox{ such that } J_1 \times J_2\in J\},\\
	\J_2=\{J_2\mid \mbox{ there is }J_1 \mbox{ ideal in } R_1 \mbox{ such that } J_1 \times J_2\in J\}.
	\end{array}
	$$
	Observe that $\J_1,\J_2$ are generators of $R_1,R_2$ respectively. Take $N$ such that $\Pi_{|i|\leq N}\alpha_1^{-i}(\I_1)\prec\J_1$ and 
	$\Pi_{|i|\leq N}\alpha_2^{-i}(\I_2)\prec\J_2$.
	It is easy to check that $\Pi_{|i|\leq N}(\alpha_1\times\alpha_2)^{-i}(\I)\prec\J$. 
	This proves that $\alpha_1\times\alpha_2$ is expansive.
	
	To prove the converse, suppose that $\I$ is an expansive generator for $\alpha_1\times\alpha_2$ 
	and consider the family of ideals
	$$\I_1=\{I_1\mid \mbox{ there is }I_2 \mbox{ ideal in } R_2 \mbox{ such that } I_1 \times I_2\in J\}.$$
	To prove that $\I_1$ is an expansive generator for $\alpha_1$ consider a generator $\J_1$ of $R_1$. 
	For the generator $\J=\{J\times R_2\mid J\in \J_1\}$ there is 
	$N$ such that $\Pi_{|i|\leq N}(\alpha_1\times\alpha_2)^i(\I)\prec\J$. 
	This implies that $\Pi_{|i|\leq N}\alpha_1^i(\I_1)\prec\J_1$ and $\alpha_1$ is expansive.
\end{proof}

\begin{prop}\label{pcociente}
	Let $R$ be a ring and $J\subseteq R$ an ideal. Call $R'$ the quotient ring $\frac{R}{J}$. For $\alpha\colon R\to R$ 
	an automorphism such that $\alpha(J)=J$, call $\alpha'\colon R'\to R'$ the induced automorphism. 
	If $\alpha$ is (positive) expansive then $\alpha'$ is (positive) expansive.
\end{prop}
\begin{proof}
	Let $\pi\colon R\to R'$ be the quotient map.
	We will show that if 
	$\I$ is an $\alpha$-generator of expansivity then $\I'=\{\pi(I)\mid I\in \I\}$ is an $\alpha'$-generator of expansivity.
	It is clear that $\I'$ is a generator of $R'$
	and that any generator of $R'$ is obtained in this way. 
	Given a generator $\K'$ of $R'$ consider $\K=\{\pi^{-1}(I')\mid I'\in\K'\}$.
	Take $N$ such that 
	$\Pi_{|i|\leq N}\alpha^{-i}(\I)\prec \K$.
	This 
	implies that
	$\Pi_{|i|\leq N}\alpha'^{-i}(\I')\prec \K'.$
	For positive expansivity the proof is similar.
\end{proof}

%
%

\section{Strong forms of expansivity}
\label{secStrongExp}
In this section we present the main results of this paper. For the proofs of Theorems \ref{teoChar0Exp} and \ref{teoPosExpFinMax} 
we introduce some definitions and a lemma.

The \emph{radical} of an ideal $I$ is the set 
$\rad{I}=\{x\in R\mid x^n\in I \mbox{ for some integer } n\}$. 
Given a set of ideals $\I$ define $\rad{\I}=\{\rad{I}\mid I\in\I\}$.

\begin{lem}
\label{lemGenAcotFinMax}
Suppose that there is $r\in\N$ such that for every generator $\K$ 
there is a generator $\J$ such that $\card(\J)\leq r$ and $\J^N\prec\K$ for some $N\geq 1$.
Then $R$ has at most $r$ maximal ideals.
\end{lem}

\begin{proof}
Arguing by contradiction, suppose that there are $r+1$ different maximal ideals 
$\m_1, \m_2,\dots, \m_{r+1}$. Define $K_i=\Pi_{j\neq i} \m_i$. 
By induction in $r$, we can prove that 
$\sum_{j\neq i}K_i=\m_i$,
which implies that $\K=\{K_1,\cdots,K_r, K_{r+1}\}$ is a generator. 
Note that no proper subset of $\K$ generates. 

Take $N\geq 1$ and a generator $\J$ of cardinal $r$ such that $\J^N\prec\K $. 
Since $\J^N\prec\K$, for each $J\in \J$ there is $K_J\in\K$ such that $J^N\subseteq K_J$.
This implies that $\rad{J}=\rad{J^N}\subseteq \rad{K_J}$, and $\rad{\J}\prec\rad{\K}$.
Since the cardinality of $\rad{\J}$ is at most $r$, 
there is a proper subset $\K'$ of $\K$ such that $\rad{\K'}$ generates. 
By \cite[Proposition 5.1 (ii) and (iii)]{L} we have that $\K'$ generates.
This contradiction proves the result.
\end{proof}

\subsection{Minimal generators}

We say that a generator $\I$ is $\prec$-\emph{minimal} if $\I\prec\J$ for every generator $\J$. In terms of expansivity, the existence of a $\prec$-minimal generator of $R$ can be seen as a 0-\emph{expansivity} of any automorphism of $R$. Indeed, if $\I$ is a minimal generator of $R$, then
$\Pi_{i=0}\alpha^i(\I)\prec\J$ for every generator $\J$ (and every automorphism $\alpha$).
In particular, the existence of a $\prec$-minimal generator gives the positive expansivity of every automorphism.

By Proposition \ref{propGenBasic}, if $R$ is local then $\I=\{R\}$ is a $\prec$-minimal generator. 
Also, a finite product of local rings has a $\prec$-minimal generator. The purpose of this section 
is to prove this statement and its converse. 

We say that a generator is \emph{strong minimal} 
if it is $\prec$-minimal and no proper subset generates. 
It is clear that every $\prec$-minimal generator contains a strong minimal generator 
and that a strong minimal generator is unique. 

\begin{teo}
\label{teoChar0Exp}
A ring $R$ admits a $\prec$-minimal generator if and only if it is a finite product of local rings.
In this case, there are exactly $k$ maximal ideals in $R$ and the strong minimal generator $\I=\{I_1,\cdots,I_k\}$ is such that:
\begin{enumerate}
\item each ideal in $\I$ is idempotent and principal,
\item $\I$ is orthogonal, that is, $I_iI_j=0$ for $i\neq j$.
\end{enumerate}
\end{teo}

\begin{proof}
We start assuming that $R$ is the product of the local rings $R_1,\dots,R_n$. 
Consider the generator of $R$ given by
	$$
	\I=\{(0,\cdots,0,R_i,0,\cdots,0)\mid i\in\{1,\dots,n\}\}.
	$$
By Proposition \ref{propGenBasic} and Remark \ref{rmkIdealProd} it is easy to check that $\I$ is a strong $\prec$-minimal generator of $R$.

To prove the direct part assume that 
$\I$ is a strong minimal generator. 
We start showing that for each $r=1,\dots,k$, there is a maximal ideal ${\mathfrak m}_r$ such that $I_r\not\subseteq {\mathfrak m}_r$ and $I_l\subseteq {\mathfrak m}_r, \forall l \neq r$. 
For $r=1,\dots,k$ consider the ideal
$\hat I_r=\sum_{i\neq r} I_i$. 
Since $\I$ has minimal cardinality we have that $\hat I_r\neq R$. 
Let ${\mathfrak m}_r$ be a maximal ideal containing $\hat I_r$. 
Since $\I$ is a generator, we have that $I_r\not\subseteq {\mathfrak m}_r$.
It is clear that ${\mathfrak m}_i\neq {\mathfrak m}_j$ when $i\neq j$. 
By Lemma \ref{lemGenAcotFinMax}, we conclude that there are no more maximal ideals.

By Proposition \ref{propGenBasic} we know that 
$\I^2$ is a generator. 
We will prove that $I_r$ is idempotent.
As $\I\prec \I^2$, we know that $I_r$ is included in some ideal of $\I^2$. 
Suppose $I_r\subseteq I_jI_k$. 
If $j\neq r$ we have $I_r\subseteq I_j$ contradicting the minimality of the cardinality of $\I$ 
(since $\J=\{I_i\mid i\neq r\}$ would be a minimal generator included in $\I$).
Then $j=r$ and similarly $k=r$. This proves that $I_r\subseteq I_r^2$.

Take $e_r\in I_r$, $r=1,\dots,k$, such that 
$1=\sum_{i=1}^ke_i$ and define $J_i=Re_i$. 
Then $\J=\{J_1,\dots,J_k\}$ is a generator and therefore $\I\prec \J$. 
For each $I_r \in I$ there is some $l$ such that $I_r\subseteq Re_l\subseteq I_l$. 
If $l\neq r$, it would contradict the strong minimality of $\I$. 
Then, $l=r$ and $I_r\subseteq Re_r\subseteq I_r$, so $I_r=Re_r$ and $I_r$ is principal.

To prove the orthogonality, for an ideal $I$ consider its annihilator 
ideal $$\ann(I)=\{a\in R\mid ax=0\ \forall x\in I\}.$$ 
As $e_i\in (Re_i)^2$ it is easy to deduce that for each $i=1,\cdots,k$ there is some $r_i\in R$ such that $e_i=r_ie_i^2$. This implies $(1-r_ie_i)\in \ann(I_i)$ and, as $r_ie_i \in I_i$ we deduce that the set $\{I_i, \ann(I_i)\}$ is a generator. 
Take $j\neq i$. 
As $\I\prec \{I_i, \ann(I_i)\}$ and $I_j\not \subseteq I_i$ (this would contradict the strong minimality of $\I$) we deduce $I_j\subseteq \ann(I_i)$ and then $I_iI_j=0$.

To finish the proof of the converse note that each $e_i$ is idempotent, since
$0=e_i(1-\sum_je_j)=e_i-\sum_j e_ie_j=e_i-e_i^2.$ This implies that each $Re_i$ can be seen as a ring with unity $e_i$, and, using orthogonality, $R=Re_1\times Re_2\times \cdots \times Re_k$.
To show that each $Re_i$ is local take $\mathfrak m_i\subseteq Re_i$ to be a maximal ideal and 
observe that $Re_1\times \cdots Re_{i-1}\times \mathfrak m_i\times Re_{i+1}\times Re_k$ is a maximal ideal in $R$. 
This gives $k$ different maximal ideals in $R$.
Assume, without loss of generality, that $Re_1$ is not local. 
Then, there is a maximal ideal $\mathfrak m'_1\neq \mathfrak m_1$. 
It is easy to show that this would give a maximal ideal $\mathfrak m'_1\times Re_2\times \cdots\times Re_k$ different from 
the $k$ ideals we had, contradicting that there are $k$ maximal ideals.
\end{proof}

\begin{rmk}
\label{rmkQ23NoMinGen}\label{elejemplo}
A ring with finitely many maximal ideals may not admit a $\prec$-minimal generator. 
Indeed, by Proposition \ref{propQ23basic} 
the ring $\Z_{2,3}$ has finitely many maximal ideals but is has no $\prec$-minimal generator 
because for every generator $\I$ there is $n$ such that $\I\nprec \{(2^n),(3^n)\}$.
This gives an example of a ring for which the identity is positively expansive (see Example \ref{partedelejemplo}) but not $0$-expansive. 
\end{rmk}

\subsection{Positively expansive automorphisms}

We show in what follows that the existence of a positive expansive automorphism on a ring implies that the ring admits finitely many maximals. 
The proof is based on \cite{AAM,AdKoMc}, but the non idempotence of algebraic generators introduces some dificulties.

\begin{teo}
\label{teoPosExpFinMax}
A ring admitting a positively expansive automorphism has finitely many maximal ideals.
\end{teo}

\begin{proof}
	Let $\I$ be a positively expansive generator of the ring and 
	define $\I_n=\Pi_{i=0}^n\alpha^{-i}(\I)$ for all $n\geq 1$.
	Since $\alpha$ is an automorphism, $\alpha(\I)$ is a generator.
	Thus, by definition, there is $N\in\N$ such that 
	 $\I_N\prec \alpha(\I)$. Applying $\alpha^{-1}$ we obtain
\begin{equation}
 \label{ecuRefinado}
 \alpha^{-1}(\I_N)\prec \I. 
\end{equation}	
Define $\J=\I_N$. We will show that 
\begin{equation}
 \label{ecuRefInduc}
 \alpha^{-n}(\J^{2^n})\prec \I_{N+n}\text{ for all }n\geq 0. 
\end{equation}
For $n=0$ it is trivial. 
Suppose that \eqref{ecuRefInduc} holds for some $n$. 
Applying $\alpha^{-1}$ to \eqref{ecuRefInduc} we get
\begin{equation}
 \label{ecuRefiUno}
 \alpha^{-n-1}(\J^{2^n})\prec \alpha^{-1}(\I_{N+n}).
\end{equation}
Since $\alpha^{-1}(\I_{N+n})\prec \alpha^{-1}(\I_N)$, by \eqref{ecuRefinado} and \eqref{ecuRefiUno} we have
\begin{equation}
 \label{ecuRefiDos}
 \alpha^{-n-1}(\J^{2^n})\prec\I.
\end{equation}
As $\I_{N+n+1}=\alpha^{-1}(\I_{N+n})\I$, 
applying Proposition \ref{propGenBasic} to \eqref{ecuRefiUno} and \eqref{ecuRefiDos}, we conclude
\[
 \alpha^{-n-1}(\J^{2^{n+1}})=[\alpha^{-n-1}(\J^{2^n})]^2\prec \I_{N+n+1}.
\]
We have proved \eqref{ecuRefInduc} by induction. 

Given any generator $\K$ if we take $n$ such that $\I_{N+n}\prec \K$ we conclude that 
$\alpha^{-n}(\J^{2^n})\prec \K$.
By Lemma \ref{lemGenAcotFinMax}
We conclude that there are finitely many maximal ideals.
\end{proof}

\begin{cor}
	If a compact metric space admits a positively expansive homeomorphism then it is finite.
\end{cor}

\begin{proof}
The maximal ideals in $C(X)$ are exactly the ideals of the form $\m_x$ for $x\in X$ (see Proposition \ref{perp} (4)). 
We deduce from Theorem \ref{teoPosExpFinMax} that $X$ has finitely many points.
\end{proof}

\subsection{Expansivity on principal ideal domains}
We recall that in a principal ideal domain, the maximal ideals are exactly the generated by irreducible elements, and these are exactly the prime elements of the ring. Moreover, every non invertible and non zero element admits a unique factorization as a product of irreducible elements. 

\begin{teo}\label{detopolandia}
If $R$ is a principal ideal domain then the following are equivalent: 
 \begin{enumerate}
 \item $R$ admits a positive expansive automorphism,
 \item the identity is expansive,
 \item $R$ has finitely many maximal ideals.
\end{enumerate}
\end{teo}

\begin{proof}
We prove first that $(1)$ implies $(3)$.
Suppose that $\alpha\colon R\to R$ is an expansive automorphism 
with generator of expansivity 
$\I=\{(f_i)\mid i=1,\dots,k\}$. Then, for any $p,q$ coprime, there is some $N\in \mathbb N$, such that
$$
\Pi_{i\leq N}\alpha^{-i}(\I)\prec\{(p),(q)\}.
$$
Note that, if $\J.\J'\prec \{(p),(q)\}$, then $J\prec \{(p),(q)\}$ or $\J'\prec\{(p),(q)\}$. Indeed, assume $\J$ does not refine $\{(p),(q)\}$, then there is some $J\in \J$ such that $J\not \subseteq (p)$ and $J\not \subseteq (q)$. But for any $J'\in \J'$ we have $JJ'\subseteq (p)$ or $JJ'\subseteq (q)$, then using that $J$ and $J'$ are principal, we deduce that every $J'$ is a subset of $(p)$ or a subset of $(q)$ and therefore $\J'\prec\{(p),)(q)\}$.

Arguing by induction, we get that if the product of finitely many generators refines $\{(p),(q)\}$ then necessarily one of them refines $\{(p),(q)\}$. We deduce that for any $p,q$ coprime, there is some $n\in \mathbb Z$ such that
$$
\alpha^{-n}(\I)\prec\{(p),(q)\}.
$$

Take some $I=(f)\in \I$ and let $X$ be the set of all irreducible elements appearing in the decomposition of $f$.
Note that $\alpha$ takes irreducibles into irreducibles, and assume there is some irreducible $p$ such that $\alpha^{r}(p) \neq p, \forall r\in \mathbb \mathbb Z$. As $X$ is finite, there is some $k$ such that $X\subseteq \{p, \alpha^{-1}(p), \cdots, \alpha^{-k}(p)\}$. Therefore, we get the contradiction that for all $n\in \mathbb N,\  \alpha^{-n}(f)\not \in \{(\alpha(p)),(\alpha^{2}(p))\}$ (otherwise $X\cap \{\alpha^{n+1}(p),\alpha^{n+2})(p)\}\neq \emptyset $).

So $p$ is periodic under applications of $\alpha$, and so is any element of $X$. Hence, the set of irreducible appearing in some $\alpha^n(f)$, with $n\in \mathbb Z$ is finite. Let us call it $\overline X$.

Arguing by contradiction, if there were infinitely many irreducible in $R$, take $s,t\not \in \overline X$ and note that $\alpha^{-n}(f)\not \in (s)$ and $\alpha^{-n}(f)\not \in (t), \forall n\in \mathbb Z$.\\

Now, for $(3)$ implies $(2)$ note first that if there are finitely many maximal ideals $\{\m_1, \m_2,\cdots, \m_r\}$ and we define $K_i=\prod_{j\neq i}\m_i$, we get that $\K=\{K_1,K_2,\cdots, K_r\}$ is a generator. Moreover, take any other generator $\J$. If $\J$ contains $R$ we are done. If not, take $p_i$ to be the irreducible generating $\m_i$, we get that $K_i$ is generated by $\prod_{j\neq i}p_i$ and also that there is some $J\in \J$ such that $J=(d)$, with $d=\prod_{j\neq i}p_j^{\beta_j}$. If $\beta$ is the maximum of the $\beta_i$'s, then $K_i^\beta \subseteq J$. Also, any crossed product $K_1K_2$ being generated by a multiple of all the $p_i$'s, it has a power contained in any ideal $J$. The product of three or more $K_i$'s is included in some product of two of them. As $K$ is finite, taking $N$ big enough, we get $\K^N\subseteq \J$.
\end{proof}
\begin{cor}
If $F$ is a field, then $F[x]$ does not admit  positive expansive automorphisms.
\end{cor}

\begin{proof}
When $F$ is a field, $F[x]$ is a principal ideal domain. 
As each ideal $(x-a)$, $a\in F$, is maximal, it follows from Theorem \ref{detopolandia} that $F$ can not be infinite.
Now, for $F$ finite, there is, for each $n\in \mathbb N$, at least one irreducible polynomial of degree $n$ 
(see \cite[Corollary 2, \S 4.13]{J}). Thus, again Theorem \ref{detopolandia} gives that there is no positive expansive automorphism of $F[x]$.
\end{proof}

\section{Spectral expansivity}
\label{secZ}
In this section we consider topological expansivity on the prime spectrum of a commutative ring, with respect to the Zariski topology. 

\subsection{Zariski topology}
\label{secZariski} Given a commutative unital ring $R$, 
the \emph{spectrum} is denoted by $\spec(R)$ and defined as follows: it is the set of all prime ideals of $R$ endowed with the topology (known as \emph{Zariski topology}) whose open sets are the sets $U_I$ consisting of all prime ideals not containing a given ideal $I$. 
It is known that $\spec(R)$ is a compact $T_0$ topological space \cite{L}.
In fact, $\spec\colon Ring \to Top$ is a functor taking a morphism of rings $\alpha\colon R\to R$ into the continuos function
$$
\spec(\alpha): \spec(R)\to \spec(R),
$$
defined by $\spec(\alpha)(\mathfrak p)=\alpha^{-1}(\mathfrak p)$.
Clearly, if $\alpha$ is an automorphism, then $\spec(\alpha)$ is a homeomorphism with $\spec(\alpha)^{-1}=\spec(\alpha^{-1})$.
We will compare topological expansivity on 
$\spec(R)$ with algebraic expansivity on $R$. 
\begin{rmk}
\label{rmkZEquiv}
By \cite[Proposition 5.1 (ii)]{L} we know that if $\{I_i\}$ is a family of ideals 
then $U_{\sum I_i}=\cup U_{I_i}$. 
Thus, $\I$ is a generator of $R$ if and only if $\{U_I\mid I\in\I\}$ is an open cover of $\spec(R)$.
\end{rmk}

\begin{prop}
\label{propZMin}
If $R$ has a $\prec$-minimal generator then $\spec(R)$ has a $\ll$-minimal cover.
If $\alpha\colon R\to R$ is an expansive automorphism then 
$\spec(\alpha)$ is an expansive homeomorphism.
Positive expansivity of $\alpha$ gives positive expansivity of $\spec(\alpha^{-1})$. 
\end{prop}

\begin{proof}
We give the details for the case of $\alpha$ an expansive automorphism, the other cases are analogous.
Suppose that $\I$ is an expansive generator of $R$ for $\alpha$. 
We will prove that $\U=\{U_I\mid I\in \I\}$ is an expansive cover of $\spec(R)$ for $\spec(\alpha)$. 
By Remark \ref{rmkZEquiv} we know that $\U$ is an open cover. 
Let $\V$ be any open cover of $\spec(R)$. 
By Remark \ref{rmkZEquiv} there is a generator $\J$ such that 
$\V=\{U_{J}\mid J\in\J\}$.
From the expansivity of $\alpha$ there is $N$ such that 
\begin{equation}
\label{ecuZcub}
\Pi_{|i|\leq N}\alpha^{-i}(\I)\prec\J
\end{equation}
Let $h=\spec(\alpha)$. We will show that $\wedge_{|i|\leq N}h^{-i}(\U)\prec\V$.
Consider $U_i\in\U$ for $|i|\leq N$. 
Take $I_i\in\I$, $|i|\leq N$, such that $U_i=U_{I_i}$ for all $|i|\leq N$.
By \eqref{ecuZcub} there is $J\in\J$ such that 
$\Pi_{|i|\leq N}\alpha^{-i}(I_i)\subset J$.
By \cite[Proposition 5.1]{L} we conclude that
$$\cap_{|i|\leq N}h^{-i}(U_i)
=\cap_{|i|\leq N}U_{\alpha^{i}(I_i)}
=U_{\Pi_{|i|\leq N}\alpha^{i}(I_i)}\subset U_{J}.$$
This proves that $\wedge_{|i|\leq N}h^{-i}(\U)\prec\V$, so $h=\spec(\alpha)$ is expansive. 
%
\end{proof}

\begin{rmk}
The converse of Proposition \ref{propZMin} is false, at least the part of $\prec$-minimal generators.   
By Proposition \ref{propQ23basic}, we have that $\spec(\Z_{2,3})$ is finite. Hence, it admits a $\prec$-minimal open cover.
As explained in Remark \ref{rmkQ23NoMinGen}, $\Z_{2,3}$ has no $\prec$-minimal generator. 
\end{rmk}

\subsection{Extension closed subsets}\label{secExtClos}
As we have mentioned, expansivity on metric spaces is preserved by restriction to closed sets. To extend this result, we introduce the following notion. A subset $Y$ of a topological space $X$ is \emph{extension closed} if for every open cover $\{U_1,\dots,U_n\}$ of 
$Y$ there is an open cover $\{V_1,\dots,V_n\}$ of $X$ such that $U_i=Y\cap V_i$ for all $i=1,\dots,n$.
In \cite[Proposition 3.12]{AAM} it is shown that 
if $h\colon X\to X$ is expansive and $Y\subset X$ is extension closed and $h(Y)=Y$ then
$h\colon Y\to Y$ is expansive.

Let $\specm(R)=\{\mathfrak p\in\spec(R)\mid \mathfrak p \text{ is maximal}\}$.

\begin{prop}
The subspace $\specm(R)$ is extension closed in $\spec(R)$.
\end{prop}

\begin{proof}
Let $\V$ be an open cover of $\specm(R)$. 
For each $V\in\V$, there is an ideal $I$ such that $V=U_I\cap\specm(R)$. 
Let $\I$ be the set of such ideals $I$ and define 
$\U=\{U_I\mid I\in\I\}$. 
To prove that $\U$ is an open cover we will show that $\I$ generates. 
Since $\V$ covers $\specm(R)$ we know that for each maximal ideal $\m$ there is $I\in\I$ such that 
$\m\in U_I$. This means that $I$ is not contained in $\m$.
If $\I$ does not generate, then there is a maximal ideal $\m_*$ such that 
$\sum_{I\in\I}I\subseteq \m_*$. This contradiction proves the result.
\end{proof}

This result and Proposition \ref{propZMin} imply that if $\alpha$ is an expansive automorphism of $R$ then $\spec(\alpha)$ is expansive as a homeomorphism in $\spec(R)$ and as a homeomorphism in $\specm(R)$.

\begin{rmk}
 If $R$ is a principal ideal domain, then $\specm(R)$ has the cofinite topology. 
 In \cite{AAM} it is shown that if a topological space has the cofinite topology and admits an expansive homeomorphism 
 then it is finite. This result is related to Proposition \ref{detopolandia}.
\end{rmk}

\begin{rmk}
 The space $\specm(R)$ is $T_1$. 
 Then, if $\alpha\colon R\to R$ is a positively expansive automorphism we have that 
 $\spec(\alpha)$ restricted to $\specm(R)$ is positively expansive. 
 By Proposition \ref{propUtzT1} we conclude that $\specm(R)$ is finite.
 This is another proof of Theorem \ref{teoPosExpFinMax}.
\end{rmk}

\subsection{Spectral spaces}
\label{secSpSp}
A topological space $X$ is a \emph{spectral space }\cite{Hoch} if it is $T_0$, compact\footnote{A topological space is \emph{compact} if every open cover admits a finite subcover. 
For reader's convenience we indicate that in \cite{Hoch} this condition is called \emph{quasi-compact}.}, the compact open
subsets are closed under finite intersection and form an open basis, and every
nonempty irreducible closed subset has a generic point.
A closed set $A$ is \emph{irreducible} if given closed sets $B,C$ such that $A=B\cup C$ then $B=A$ or $C=A$. 
A point $x\in A$ is \emph{generic} if $\clos\{x\}=A$.

Let us give an example of a spectral space. 
Consider the funcion $h\colon \R\to \R$ defined as $h(x)=\sqrt[3]{x}$ 
and define 
\[
 X=\{-1,0,1\}\cup\{h^i(1/2)\mid i\in\Z\}\cup\{h^i(-1/2)\mid i\in\Z\}.
\]
On $X$ we consider the topology $\tau=\{(a,b)\cap X: a<0<b\}\cup\{\emptyset\}$. 
Notice that every open cover contains two open sets of the form 
$[-1,b)\cap X$ and $(a,1]\cap X$, which implies that $X$ is compact. 
It is clear that it is $T_0$. 
To prove that it is a spectral space, notice that if $-1<a<0<b<1$ then $(a,b)\cap X$ is open and compact. 
Thus, the compact-open subsets form a basis of the topology, closed under finite intersections. 
Finally, the irreducible closed subsets are 
$X\cap [-1,a]$ for some $a\in X\cap [-1,0)$ and 
$X\cap [b,1]$ for some $b\in X\cap (0,1]$. 
Notice that $X\cap [-1,a]=\clos\{a\}$ 
and $X\cap [b,1]=\clos\{b\}$.


\begin{exa}
\label{exaX3Exp} 
In the space $X$ defined above it holds that: 
\begin{enumerate}
 \item the identity is not expansive,
 \item $h(x)=\sqrt[3]{x}$ is positively expansive,
 \item its inverse, $h^{-1}(x)=x^3$ is not positively expansive.
\end{enumerate}
The identity is not expansive because $X$ has no $\ll$-minimal open cover. 
The homeomorphism $h(x)=\sqrt[3]x$ is positively expansive with expansivity cover
$\U=\{[-1,1/2)\cap X,(-1/2,1]\cap X\}$.
It is easy to see that $h^{-1}$ is not positively expansive.
\end{exa}

\begin{rmk}
Last example shows that for a ring with finitely many maximal ideals the identity may not be positively expansive. Indeed,
Let $X$ be the spectral space defined above. 
By \cite{Hoch}, there is a ring $R$ such that $\spec(R)$ is homeomorphic to $X$. 
Since $X$ has two minimal closed sets, $R$ has two maximal ideals. 
Since the identity of $X$ is not positively expansive we conclude that the identity of $R$ is not positively expansive.
\end{rmk}

It would be interesting to characterize the objects (rings and topological spaces) admitting positively expansive automorphisms.


\begin{thebibliography}{99}

\bibitem{AAM} M. Achigar, A. Artigue and I. Monteverde,
\textit{Expansive homeomorphisms on non-Hausdorff spaces},
Topology and its Applications,
\textbf{207}, 
(2016),
109--122.

\bibitem{AdKoMc} R.L. Adler, A.G. Konheim and M.H. McAndrew, 
\textit{Topological entropy},
Trans. Amer. Math. Soc., \textbf{114},
(1965), 
309--319.

\bibitem{AH} N. Aoki and K. Hiraide,
\textit{Topological theory of dynamical systems},
North-Holland,
(1994).

\bibitem{CK}
E.M. Coven and M. Keane,
\textit{Every compact metric space that supports a positively expansive homeomorphism is finite},
IMS Lecture Notes Monogr. Ser., Dynamics \& Stochastics, \textbf{48}, (2006), 304--305.

\bibitem{GJ} L. Gillman and M. Jerison, \textit{Rings of Continuous Functions},60).

\bibitem{Hoch} M. Hochster, \textit{Prime Ideal Structure in Commutative Rings}, 
Trans. Amer. Math. Soc., \textbf{142}, (1969), 43--60.

\bibitem{H} T.W. Hungerford,
\textit{Algebra},
Springer-Verlag New York, Inc.,
(1974).

\bibitem{J} N. Jacobson, \textit{Basic Algebra I}, Dover Publications, second edition, (2009).

\bibitem{K} J.L. Kelley, \textit{General Topology}, 
Springer, (1955).

\bibitem{KR} H. Keynes and J. Robertson, \textit{Generators for topological entropy and expansiveness},
Mathematical systems theory,
\textbf{3},
(1969),
51--59.

\bibitem{L} S. Lang, \textit{Algebra}, Springer, 3rd ed., (2002).
\end{thebibliography}
\end{document}